\RequirePackage{fix-cm}
\documentclass[smallextended]{svjour3}       
\smartqed  

\usepackage{amsmath, amsfonts, vmargin, enumerate}
\usepackage{amssymb}

\newcommand{\real}{{\mathbb R}}

\newcommand{\A}{{\mathcal A}}

\newcommand{\E}{{\mathcal E}}
\newcommand{\F}{{\mathcal F}}

\newcommand{\M}{{\mathcal M}}
\newcommand{\N}{{\mathcal N}}

\newcommand{\D}{{\mathcal D}}

\newcommand{\e}{\varepsilon}

\newcommand{\8}{\infty}

\newcommand{\be}{\begin{eqnarray*}}
\newcommand{\ee}{\end{eqnarray*}}
\newcommand{\beq}{\begin{equation}}
\newcommand{\eeq}{\end{equation}}
\newcommand{\beqn}{\begin{equation*}}
\newcommand{\eeqn}{\end{equation*}}
\newcommand{\bsp}{\begin{split}}
\newcommand{\esp}{\end{split}}

\begin{document}

\title{On products of noncommutative symmetric quasi Banach spaces and applications
}


\author{Turdebek N. Bekjan         \and
        Myrzagali N. Ospanov 
}


\institute{T. N. Bekjan  \at
              College of Mathematics and Systems Science, Xinjiang
University, Urumqi 830046, China \\
                           \email{bekjant@yahoo.com \\ bek@xju.edu.cn}           
           \and
           M. N. Ospanov \at
              Faculty of Mechanics and Mathematics, L. N. Gumilyov Eurasian National University, Nur-Sultan 010008, Kazakhstan\\
                \email{myrzan66@mail.ru}
}

\date{Received: date / Accepted: date}

\maketitle

\begin{abstract}

Let $E_1,\;E_2$ be symmetric quasi Banach function spaces on $(0,\alpha)\;(0<\alpha\le\8)$. We study some properties of several constructions (the products $E_1(\M)\odot E_2(\M)$, the Calder$\rm\acute{o}$n spaces $E_1(\M)^\theta E_2(\M)^{1-\theta}$, the complex interpolation spaces $(E_1(\M),E_2(\M))_\theta$, the real interpolation method $(E_1(\M),E_2(\M))_{\theta,p}$) in the context of noncommutative symmetric quasi Banach spaces. Under some natural assumptions, we prove
$$
(E_1(\M), E_2(\M))_\theta=E_1(\M)^\theta E_2(\M)^{1-\theta}=E_1^{(\frac{1}{\theta})}(\M)\odot E_2^{(\frac{1}{1-\theta})}(\M)\;(0<\theta<1).
$$
As application, we  extend these result to  the noncommutative symmetric quasi Hardy spaces case. We also obtained the real case of Peter Jones' theorem   for  noncommutative symmetric quasi Hardy spaces.
\end{abstract}


\section{Introduction}

The factorization of Banach function spaces has been intensively developed by many authors during last decades. The Lozanovski$\rm\breve{i}$ factorization theorem (see \cite{Lo}) has a lot of important generalizations in which the spaces of pointwise products $E \odot F$ and the spaces of pointwise multipliers $M (E,F)$ play a crucial role (see \cite{KLM} for more bibliography, also see \cite{Sch}).

 In \cite[Theorem1]{KLM}, the authors proved that for symmetric Banach function spaces, $E_1\odot E_2$  is the $\frac{1}{2}$-concavification of the Calder$\rm\acute{o}$n space $E_1^{1/2}E_2^{1/2}$, i.e.,
  $$
  E_1\odot E_2=(E_1^{1/2}E_2^{1/2})^{(\frac{1}{2})}
  $$
 (see \cite{Ca} for definition of the Calder$\rm\acute{o}$n space $E_1^{1-\theta}E_2^{\theta}\;(0<\theta<1)$). This relationship also holds for the symmetric quasi Banach function spaces case (It can be done exactly the same way as in \cite[Theorem1]{KLM}).

Let $E_1$ and $E_2$ be  symmetric Banach function spaces on $(0,\alpha)$  and $0<\theta<1$. From  Theorem IV.1.14 of \cite{KPS}
(see also \cite{Ca} ) it follows that if $E_1^{1-\theta}E_2^\theta$ has order continuous norm,  then
\be
(E_1,E_2)_\theta=E_1^{1-\theta}E_2^\theta
\ee
holds with equality of norms. Under some  assumptions, this result holds for the symmetric quasi Banach function  spaces case (cf. \cite[Theorem 3.4]{KM} and \cite[Theorem 7.9]{KMM}).

The first goal of the paper is to study the products $E_1(\M)\odot E_2(\M)$, the Calder$\rm\acute{o}$n spaces $E_1(\M)^\theta E_2(\M)^{1-\theta}$, the complex interpolation spaces $(E_1(\M),E_2(\M))_\theta$ and the relations between them for the case of quasi Banach
function spaces. We prove if  $E_j$ is a symmetric quasi Banach function   space on $(0,\alpha)$ with order continuous norm which is
$s_j$-convex for some $0 < s_j < \8$, $j=1,\;2$ and $0<\theta<1$, then
\beq\label{eq:interpolation-calderon-product}
(E_1(\M),E_2(\M))_\theta=E_1(\M)^\theta E_2(\M)^{1-\theta}=E_1^{(\frac{1}{\theta})}(\M)\odot E_2^{(\frac{1}{1-\theta})}(\M),
\eeq
where $\mathcal{M}$ is a semifinite  von Neumann algebra with a faithful normal semifinite  trace
$\tau$ satisfying $\tau(1)=\alpha$. Applying  \eqref{eq:interpolation-calderon-product},  we give similar result for the case noncommutative symmetric quasi Hardy spaces associated with a finite
subdiagonal algebra in Arveson's sense \cite{A}. Let $\M$ be a finite von Neumann algebra  with a faithful normal finite  trace
$\tau$ satisfying $\tau(1)=\alpha$, and let $\A$ be a subdiagonal algebra of $\M$.  If $E_j$ is  a symmetric quasi Banach function space on $(0,\alpha)$ with order continuous norm which is
$s_j$-convex for some $0 < s_j < \8$ $(j=1,\;2)$ and
 $q_{E_1},\; q_{E_2}<\8$, then
\beq\label{eq:interpolation-calderon-hardy}
(E_1(\A),E_2(\A))_\theta=E_1(\A)^\theta E_2(\A)^{1-\theta}=E_1^{(\frac{1}{\theta})}(\A)\odot E_2^{(\frac{1}{1-\theta})}(\A)\qquad (0<\theta<1)
\eeq
with equivalent norms.

 Pisier\cite{P} gave a new proof of the interpolation theorem
of Peter Jones (see \cite{Jo} or \cite{BS}, p.414). His method does to extend to the noncommutative case and the case of Banach space valued $H^p$-spaces  (see $\S2$ and $\S3$ in \cite{P}). In \cite{PX}, Pisier and Xu obtained
noncommutative version of P. Jones' theorem  for noncommutative Hardy spaces $H_p(\A)$ associated with a finite
subdiagonal algebra. The noncommutative version of  Peter Jones' theorem   for  noncommutative Hardy spaces  $H_p(\A)$  associated with semifinite von Neumann algebras holds  (for the real case see \cite{B,STZ}, for complex case see \cite{BO}). We use  \eqref{eq:interpolation-calderon-hardy} and Pisier's method to prove the real case of Peter Jones' theorem   for noncommutative symmetric quasi Hardy spaces.

The paper is organized as follows.
In Section $2$ some necessary definitions and notations are collected including the symmetric quasi Banach function spaces,
noncommutative symmetric quasi Banach spaces and  interpolations.
In Section $3$ the  pointwise product of symmetric quasi Banach function spaces is defined and the \eqref{eq:interpolation-calderon-product} is presented. Section 4 is devoted to complex interpolation of the noncommutative symmetric quasi Hardy spaces, the \eqref{eq:interpolation-calderon-hardy} is proved. In Section 5, we proved the real case of Peter Jones' theorem   for noncommutative symmetric quasi Hardy spaces.

In what follows, $A,\;B$ and $C$ always denote  constants, which may be different in different places.

\section{Preliminaries}\label{pre}

Let $(\Omega,\Sigma, m)$ be  a complete $\sigma$-finite measure space. We denote by $L_0(\Omega)$ the space of $\mu$-measurable real-valued functions defined  on $\Omega$. The nonincreasing  rearrangement function $f^*: [0, \8) \mapsto [0, \8]$ for $f \in L_0 (\Omega)$ is defined by
$$
f^*(t) = \inf\{s > 0 : \; \mu ( \{\omega \in \Omega :\; |f (\omega)| > s\}) \le t\}
$$
for $t \ge 0.$ If $f, g \in L_0(\Omega)$ such that $\int_{0}^{t} f^* (s) ds \le \int_{0}^{t}g^* (s) ds$ for all $t \ge 0,$ $f$ is said to be {\it majorized} by $g,$ denoted by $f \preccurlyeq g.$

Let $E$ be a (quasi) Banach space of functions in $L_0(\Omega)$. If from  $g\in E,\;f\in L_0(\Omega)$ and $|f(t)|\le|g(t)|$ for $m$-almost all $t\in\Omega$  follows that  $f\in E$ and $\|f\|_E\le\|g\|_E$, then we say $E$ is  a (quasi) Banach ideal space on $\Omega$.

Let $0<\alpha\le\8$. If $E$ is  a (quasi) Banach ideal space on  $(0,\alpha)$ and satisfying the following
 properties: if  $f\in E$ and $g\in L_0(0, \alpha)$ such that $g^* (t) \le f^*(t)$ for all $t \ge 0$,  implies that $g\in E$ and $\|g\|_{E}\le\|f\|_{E}$,
 then $E$ is called  a symmetric  (quasi) Banach function space on $(0,\alpha)$.

 For   $0<p<\infty,\; E^{(p)}$ will denote the symmetric
quasi Banach function space defined by
$$
E^{(p)}=\{f:\;|f|^{p}\in E\},
$$
equipped with the quasi norm
$$
\|f\|_{E^{(p)}}=\||f|^{p}\|_{E}^{\frac{1}{p}}.
$$
 Observe that, if $0<p,\;q < \8$, then
 \beq\label{eq:pawerp-q}
 (E^{(p)})^{(q)} = E^{(pq)}.
 \eeq
  It is to be noted that, if $E$ is a symmetric Banach function space and $p > 1$, then the space $E^{(p)}$ is a symmetric Banach function space
and is usually called the $p$-convexification of $E$ (see \cite[Proposition 3.1]{DDS} or see \cite{ORS,X}).

 A symmetric  (quasi ) Banach function space $E$ on $(0,\alpha)$ is called fully symmetric if, in addition, for $f \in L_0(0,\alpha)$ and $g \in E$
with $f\preccurlyeq g$ it follows that $f \in E$ and $\|f\|_E \le \|g\|_E$. $E$ is said to have a Fatou (quasi) norm if for every net $(f_{i})_{i\in I}$ in $E$ and $f\in E$ satisfying $0\le f_i\uparrow f$ we have $\|f_i\|_E \uparrow\|f\|_E$; $E$ is said to have the Fatou property  if, for every net $(f_{i})_{i\in I}$ in $E$ satisfying $0\le f_i\uparrow$ and $\sup_{i\in I}\|f_i\|_E < \infty,$ the supremum $f =\sup_{i\in I} f_i$ exists in $E$ and $\|f_i\|_E\uparrow\|f\|_E$;
If for every net $(f_{i})_{i\in I}$ in $E$ such that $f_i\downarrow0$,  $\|f_i\|_E\downarrow0$ holds, then we call $E$ has order continuous norm.

The K$\rm\ddot{o}$the dual of a symmetric Banach function space $E$ on $(0,\alpha)$ is the symmetric Banach function
space $E^{\times}$ given by
$$
E^{\times} =\left\{g\in  L_0(0,\alpha):\; sup\{\int_{0}^\alpha
|f(t)g(t)| dt\; : \|f\|_E \le 1\}< \infty\right\};
$$
$$
\|g\| = sup\{\int_{0}^\alpha
|f(t)g(t)| dt\; : \|f\|_E \le 1\},\quad
g \in  E^{\times}.
$$

Let $E$ be a symmetric Banach function space on $(0,\alpha)$. Then   $E$  has order continuous norm if and only if
it is separable, which is also equivalent to the statement $E^{*} = E^{\times}$ (see \cite[Chapter1, Corollary 4.3 and 5.6]{BS}).  Moreover, if $E$  is a
 separable symmetric quasi Banach function space, then  $E$ has  the Fatou property and  fully symmetric.

Let $0<s,t<\infty$. If there exists a
constant $C>0$ such that for all finite sequence $(f_{n})_{n\ge1}$ in $E$
$$
\begin{array}{c}
  \|(\sum|f_{n}|^s)^{\frac{1}{s}}\|_E\le C (\sum\|f_{n}\|_E^s)^{\frac{1}{s}}\\
(\mbox{resp}.\;\|(\sum|f_{n}|^{t})^{\frac{1}{t}}\|_E\ge C^{-1}
(\sum\|f_{n}\|_E^{t})^{\frac{1}{t}}),
\end{array}
$$
then $E$ is called  $s$-convex (resp. $t$-concave).
The least such constant $C$ is called the $s$-convexity  (resp.
$t$-concavity) constant of $E$ and is denoted by $M^{(s)}(E)$ (resp.
$M_{(t)}(E)$). If $E$ is $s$-convex and $s$-concave
then  $E^{(p)}$ is $ps$-convex and $ps$-concave with
$M^{(ps)}(E^{(p)})= M^{(s)}(E)^{\frac{1}{p}}$ and
$M_{(pt)}(E^{(p)})= M_{(t)}(E)^{\frac{1}{p}}$ (see \cite[Proposition 3.1]{DDS}). Therefore, if $E$ is
$s$-convex  then  $E^{(\frac{1}{s})}$ is $1$-convex, so it
can be renormed as a Banach lattice (see \cite[Proposition 1.d.8]{LT} and \cite[p. 544]{X}).

For any $0 < a <\8$, let the dilation operator $D_a$ on $L_0(0,\alpha)$ defined
by
$$(D_af)(\tau) =\left\{\begin{array}{ll}
f(\frac{\tau}{a})&\quad\mbox{if}\;\frac{\tau}{a}<\alpha,\\
0&\quad\mbox{if}\;\frac{\tau}{a}\ge\alpha,
\end{array}
\right.$$
for $\tau \in (0,\alpha)$. For a symmetric quasi Banach function space $E$ on $(0,\alpha)$, if $f\in E$,  then $D_af\in E$ and there exists a constant $K(a)$ such
that $\|D_af\|_E \le K(a)\|f\|_E$ for all $f\in E$ (cf. \cite[Lemma 2.2]{D}).  If $E$ is a Banach function space, then the constant $K(a)$ can be taken to be
$max\{1,a\}$ (see \cite{KPS}). The lower and upper Boyd indices $p_E$ and $q_E$ of $E$ are respectively defined by
$$
p_E = \sup_{a >1} \frac{\log a}{\log \|D_a \|_E}\quad \text{and}\quad
q_E = \inf_{0 < a <1} \frac{\log a}{\log \|D_a \|_E}.
$$
Note that $0< p_E \le q_E \le\8$ (cf. \cite[Lemma 2.2]{D}). If $E$ is a symmetric Banach function space, then $1 \le p_E \le q_E \le\8$ (cf. \cite[Proposition 2.b.2]{LT}).
If $E$ is $s$-convex then $p_E\ge s$ and if $E$ is $t$-concave
then $q_E\le t$. The proof of these facts can be done exactly the same way as in the particular case where $E$ is a Banach function
space (see \cite[p.132]{LT}).

We shall need the following duality for Boyd indices (see \cite[Theorem II.4.11]{KPS}). Let $E$ be a symmetric Banach function space on $(0,\alpha)$  with the Fatou norm. Then
\beq\label{boydindex}
\frac{1}{p_E}+\frac{1}{q_{E^\times}}= 1,\quad \frac{1}{p_{E^\times}}+\frac{1}{q_E}= 1.
\eeq
For more details on symmetric (quasi) Banach function space we refer to \cite{BS,DDS,KPS,LT}.

We use standard notation in  theory of noncommutative
$L^{p}$-spaces,  our main references are \cite{FK,PX} (see
also \cite{PX} for more  historical references). Let $\mathcal{M}$ be a semifinite  von Neumann algebra on
the Hilbert space $\mathcal{H}$ with a faithful normal semifinite  trace
$\tau$ satisfying $\tau(1)=\alpha$. The set of all $ \tau$-measurable operators will be
denoted by $L_0(\M)$. On $L_0(\M)$, we define sum ( respectively, product) by closure of the
algebraic sum (respectively, the
algebraic product), then $L_0(\M)$ is a
$\ast$-algebra.

 Let $x\in L_0(\M)$.  For $t>0$, we define
$$
\lambda_{t}(x)=\tau(e_{(t,\infty)}(|x|)),
$$
where $e_{(t,\infty)}(|x|)$ is the spectral projection of $|x|$ corresponding
to the interval $(t,\infty)$, and
 $$  \mu_t(x)=\inf\{s>0\;:\; \lambda_s(x)\le t\}.$$
The functions $t\mapsto \lambda_t(x)$ and $t\mapsto \mu_t(x)$ are called the
distribution function and  the   generalized singular
numbers  of $x$  respectively. We  denote simply by $\lambda (x)$ and  $\mu(x)$, the
distribution function and  the   generalized singular
numbers  of $x$  respectively. It is easy to check that $\mu_{t}(x)=0$, for all $t\ge\tau(1)$.

For more details on generalized singular
numbers  of measurable operators we refer to
\cite{FK}.

Let  $E$ be a symmetric quasi Banach function space on
$(0,\alpha)$. We define
$$
E(\M)=\{x  \in L_0(\M) :\;\mu(x)\in
E\};
$$
$$
\|x\|_{E}=\|\mu(x)\|_{E}, \quad \quad x
\in E(\M).
$$
Then $(E(\M),\|.\|_E)$ is a  quasi Banach
space and we call $(E(\M),\|.\|_E)$  a noncommutative symmetric quasi Banach space (cf. \cite{DDP1,Pa,S,X}).

Let $(X_0 ,X_1)$ be an interpolation couple of quasi Banach spaces, i.e. $X_j,\; j = 0,\;1$ are continuously embedded into a larger topological vector space $Y$, and
$X_0 \cap X_1$ is dense in $X_j,\; j = 0,\;1$.
Let $S$ (respectively, $\overline{S}$ ) denote the open strip $\{z:\;0<\mathrm{Re}z<1\}$ (respectively, the closed strip $\{z:\;0\le \mathrm{Re}z\le1\}$ ) in the complex plane $\mathbb{C}$.
Let $\F(X_0, X_1)$ be   the space of bounded, analytic functions $f : S \rightarrow X_0 + X_1$ which extend continuously to $\overline{S}$ such that the traces $t\rightarrow f(j+it)$ are bounded continuous functions into $X_j,\; j = 0,\;1$. We equip $\F(X_0, X_1)$ with the norm:
 $$
 \big\|f\big\|_{\F(X_0, X_1)}=\max\big\{
 \sup_{t\in\real}\big\|f(it)\big\|_{X_0}\,,\;
 \sup_{t\in\real}\big\|f(1+it)\big\|_{X_1},\;\sup_{z\in S}\big\|f(z)\big\|_{X_0+X_1}\big\}.
 $$
Then  $\F(X_0, X_1)$ becomes
a quasi Banach  space.   We define the complex interpolation space $(X_0,X_1)_\theta \; (0<\theta<1)$
by $x\in (X_0,X_1)_\theta$ if and only if $x\in\F(X_0, X_1)(\theta)$ and
 $$\|x\|_{(X_0,\; X_1)_\theta}=\inf\big\{\big\|f\big\|_{\F(X_0, X_1)}\;:\; f(\theta)=x,\;
 f\in\F(X_0, X_1)\big\}.$$
It then follows that $ (X_0,X_1)_\theta$ is a quasi Banach space for $0<\theta<1$ (see \cite{KM}).

It should be pointed out that the third term is superfluous when  the maximum modulus
principle  applies. It is well-known that if $(X_0, X_1)$ is a compatible couple of complex Banach spaces, then the definition of the interpolation spaces here coincides with that of \cite{Ca}.

  However, the maximum  principle  fails for some quasi
Banach spaces (see \cite{Pe1}), but there
is an important subclass of quasi Banach spaces called analytically convex in \cite{K} in which
the maximum modulus principle does hold. A quasi Banach space $X$ is analytically convex if there is a constant $C$ such that for every
polynomial $P:\mathbb{C}\rightarrow X$ we have $\|P(0)\|_X \le C\max_{|z|=1} \|P(z)\|_X $.  The analytically convex property is
equivalent to the condition that
\beq\label{eq:kalton}
\sup_{z\in S}\big\|f(z)\big\|_{X}\le C \max\{
 \sup_{t\in\real}\big\|f(it)\big\|_{X}\,,\;
 \sup_{t\in\real}\big\|f(1+it)\big\|_{X}\}
\eeq
for any analytic function $f: S\rightarrow X$ which is continuous on the closed strip
$\overline{S}$ (see \cite{KMM}).

Let us recall  the real interpolation method.  Let $X_0, X_1$ be a compatible couple of quasi Banach  spaces.
For all $x \in X_0+X_1$ and for all $t>0$, we let
$$
K_t (x, X_0, X_1)=\inf
\{\|x_0\|_{X_0}+t\|x_1\|_{X_1}:\; x = x_0 + x_1,\;x_0\in X_0,\; x_1\in X_1 \}.
$$
For all $x \in X_0\cap X_1$ and for all $t>0$, we let
$$
J_t (x, X_0, X_1)=\max
\{\|x\|_{X_0},\;t\|x\|_{X_1}\}.
$$
Recall that the (real interpolation) space $(X_0, X_1)_{\theta, p}$ is defined as the space of all $x$ in
$X_0+X_1$ such that $\|x\|_{\theta, p} < \8$, where
$$
\|x\|_{\theta, p}  = (\int_0^\8(t^{-\theta}K_t (x, X_0, X_1))^p\frac{dt}{t})^{\frac{1}{p}}
$$
($0<\theta<1$, $0<p\le\8$).
Recall that there is a parallel definition $(X_0, X_1)_{\theta, p}$ using the $J_t$ functional which
leads to the quasi same Banach space with an equivalent norm. See \cite{BS,BeL} for more information on $K_t$ and $J_t$  functionals.

\section{On product of symmetric quasi Banach spaces}

If $E$  is  a (quasi) Banach ideal space on  $(0,\alpha)$, we  set
$$
E^+=\{f:\;f\in E,\;f\ge0\;\mbox{ a.e. on}\; (0,\alpha)\},\quad B(E)=\{f:\;f\in E,\;\|f\|_E\le1\}.
$$

Let  $E_{i}$  be  a (quasi) Banach ideal space on
$(0,\alpha),\; i=1, 2$. We define the pointwise product space $E_1\odot E_2$ as
\begin{equation}\label{multiplicativityspace}
    E_{1}\odot E_{2}=\{f:\;f=f_{1}f_{2},\;f_{i}\in E_{i},\;i=1,\;2\}
\end{equation}
with a functional $\|f\|_{E_1\odot E_2}$ defined by
$$
\|f\|_{E_1\odot E_2}=\inf\{\|f_{1}\|_{E_1}\|f_{2}\|_{E_2}:\;f=f_{1}f_{2},\;f_{i}\in E_{i},\;i=1,\;2\}.
$$
Similar to Proposition 1 in \cite{KLM}, for any $f\in E_1\odot E_2$,  we deduce that
\beq\label{eq:product-norm}
 \|f\|_{E_1\odot E_2} =\inf\{\|f_{1}\|_{E_1}\|f_{2}\|_{E_2}:\;|f|\le f_{1}f_{2},,\;f_{i}\in E_{i}^+,\;i=1,\;2\}.
\eeq

Let  $E_{i}$ be a (quasi) Banach ideal space on
$(0,\alpha),\; i=1,\; 2$, and let $0<\theta<1$. We define the Calder$\rm\acute{o}$n space $E_1^{1-\theta} E_2^{\theta}$ as
\be
E_1^{1-\theta} E_2^{\theta}=\{f:\;|f|\le \lambda|f_1|^{1-\theta} |f_2|^{\theta},\;\lambda>0,\;\;f_{i}\in B(E_{i}),\;i=1,\;2\}.
\ee
We equip this space with the  functional
 $$
 \|f\|_{E_1^{1-\theta} E_2^{\theta}}= \inf \{\lambda \}
 $$
 where the
infimum is over all such representations.

If   $E_1,\;E_2$ are symmetric   Banach  function spaces on  $(0,\alpha)$,
by  \cite[Lemma II.4.3]{KPS},  $ E_1^{1-\theta} E_2^{\theta}$ is again a symmetric Banach function
space on $(0,\alpha)$.   This result also holds for quasi Banach ideal spaces. We first note that it
follows directly  the same way as in \cite[p. 243-244]{KPS}  that we have:

\begin{proposition}\label{pro:calderon}
Let  $E_1$ and $E_2$ be   a quasi Banach ideal spaces on
$(0,\alpha)$, and let $0<\theta<1$. Then the Calder$\rm\acute{o}$n space $E_1^{1-\theta} E_2^{\theta}$ is a quasi Banach ideal space on
$(0,\alpha)$.
\end{proposition}
The whole proof of \cite[Theorem 1]{KLM} works for quasi Banach
ideal spaces, hence, we have  the following:
\begin{proposition}\label{thm:product-calderon}
Let  $E_{i}$ be a symmetric quasi  Banach function space on
$(0,\alpha),\; i=1,\; 2$. If $0<\theta<1$ and $0<p<\8$, then
 \begin{enumerate}[\rm(i)]
            \item $E_1^{1-\theta} E_2^{\theta}=E_1^{(\frac{1}{1-\theta})}\odot E_2^{(\frac{1}{\theta})}$,
          \item $ (E_{1}\odot E_{2})^{(p)}=E_1^{(p)}\odot E_2^{(p)}$,
          \item $ E_{1}\odot E_{2}=(E_1^\frac{1}{2} E_2^\frac{1}{2})^{(\frac{1}{2})}$.
 \end{enumerate}
\end{proposition}

\begin{proposition}\label{pro:product}
Let  $E_{j}$ be a symmetric quasi Banach function space on
$(0,\alpha)$, $j=1,\;2$. Then the Calder$\rm\acute{o}$n space $(E_1^{1-\theta} E_2^{\theta},\|\cdot\|_{E_1^{1-\theta} E_2^{\theta}})$ can be equipped with
an equivalent  quasi norm $\|\cdot\|$ so that $(E_1^{1-\theta} E_2^{\theta},\|\cdot\|)$ is a symmetric quasi Banach function space  on
$(0,\alpha)$.
\end{proposition}
\begin{proof} Let $F=(E_1^{1-\theta} E_2^{\theta},\;\|\cdot\|_{E_1^{1-\theta} E_2^{\theta}})$. We denote by $(F^{(*)},\|\cdot\|_{F^{(*)}})$ the symmetrization of $F$, where $\|f\|_{F^{(*)}}=\|f^*\|_{F}$ and  $f^*$ is the nonincreasing  rearrangement of $f$. It is enough to prove that  $(F^{(*)},\|\cdot\|_{F^{(*)}})$ is a quasi Banach function space and $F=F^{(*)}$ with equivalent quasi norms.

(i) If $f\in F$, then $|f|\le\lambda|f_1|^{1-\theta}|f_2|^{\theta}$ for some $\lambda>0$ and $f_i\in B(E_i),\; i=1,\;2$. Applying the property $10^\circ$ from \cite[p.67]{KPS}
and the fact that $(|g|^p)^*=(g^*)^p\;(g \in L_0 (\Omega))$, we get
 $$\begin{array}{rl}
    f^*(t)& \le \big(\lambda|f_1|^{1-\theta}|f_2|^{\theta}\big)^{*}(t) \le\lambda\big(|f_1|^{1-\theta}\big)^{*}(\frac{t}{2})\big(|f_2|^{\theta}\big)^{*}(\frac{t}{2}) \\
      & =\lambda\big[D_{2}f_1^{*}\big]^{1-\theta}(t)\big[D_{2}f_2^{*}\big]^{\theta}(t).
        \end{array}
 $$
 Since $D_{2}$ is bounded on each $E_i$ (cf. \cite[Lemma 2.2]{D}), we conclude $f^*\in F$.
 It follows that $F\subset F^{(*)}$ and there is a constant $C>0$ such that
$$
 \|f\|_{F^{(*)}}\le C\|f\|_F,\qquad \forall f\in F.
$$

 (ii) Let  $f^*\in F$. Then $f^*\le\lambda|f_1|^{1-\theta}|f_2|^{\theta}$ for some $\lambda>0$ and $f_i\in B(E_i),\; i=1,\;2$.

 First suppose that $f^*(\8)=0$. Then there is a measure preserving transformation $\sigma$ such that $f^*\circ\sigma=|f|$ (see \cite{BS} for example). Thus
 $$
|f|=f^*\circ\sigma\le\lambda|f_1\circ\sigma|^{1-\theta}|f_2\circ\sigma|^{\theta}
 $$
 and $f_i\circ\sigma\in B(E_i),\; i=1,2$, because $E_i$ are symmetric. Therefore $|f|\in F$. Since $F$ has the ideal property (see \cite[Proposition 1]{KLM}), so $f\in F$.

 Now suppose that $f^*(\8)>0$.  Then there is a measure preserving transformation $\sigma:\mathcal{C}\rightarrow [0,\8)$ such that $f^*\circ\sigma=|f|$ a.e. on $\mathcal{C}$, where $\mathcal{C}=\{t:|f(t)|>f^*(\8)\}$ ( see \cite[Lemma 2.2]{CKP}  and \cite{Ka}). Similarly as above we get $|f|\chi_\mathcal{C}\in F$. On the other hand, since $f^*(\8)>0$ and $f^*(\8)\chi_{(0,\8)}\le f^*$, so $\chi_{(0,\8)}\in F$ (because $F$ has the ideal property) and consequently $|f|\chi_{(0,\8)\backslash \mathcal{C}} \le f^*(\8)\chi_{(0,\8)}\in F$. Hence, $|f|\in F$, and so $f\in F$. From this we obtain that $ F^{(*)}\subset F$ and there exists  a constant $C'>0$ such that
$$
\|f\|_F\le C'\|f\|_{F^{(*)}},\qquad \forall f\in F^{(*)}.
$$

By (i) and (ii), it follows immediately that $F^{(*)}$ is a linear space. Using Corollary 1 in \cite{KLM1}, we obtain that  $(F^{(*)},\|\cdot\|_{F^{(*)}})$ is a quasi-normed space. Hence, from (i) and (ii), we get the desired result.
 This finishes proof.
\end{proof}

By Proposition \ref{thm:product-calderon} and \ref{pro:product}, we get the following result.

\begin{corollary}\label{cor:product}
Let  $E_{i}$ be a
symmetric quasi Banach  function space on
$(0,\alpha)$, $j=1, 2$. Then there is an equivalent quasi norm $\|\cdot \|$ so that $(E_{1}\odot E_{2},\;\|\cdot \|)$ is  a symmetric  quasi Banach function space on
$(0,\alpha)$.
\end{corollary}

Let $E_1$ and $E_2$ be  symmetric Banach function spaces on $(0,\alpha)$  and $0<\theta<1$. From \cite[Theorem IV.1.14]{KPS}
(see also \cite{Ca} ) it follows that if $E_1^{1-\theta}E_2^\theta$ has order continuous norm,  then
$(E_1,E_2)_\theta=E_1^{1-\theta}E_2^\theta$ holds with equality of norms.  For symmetric quasi Banach function spaces, we have the following result.

\begin{theorem}\label{thm:interpolation-colderon}
Let $E_j$ be a symmetric quasi Banach function space on $(0,\alpha)$ which is
$s_j$-convex for some $0 < s_j < \8$, $j=1,\; 2$. If  $E_j$ has order continuous norm, $j=1,\; 2$, then
\be
(E_1,E_2)_\theta=E_1^{1-\theta}E_2^\theta,\quad \forall \theta\in(0,1).
\ee
and
$$
A\|x\|_{ E_1^{1-\theta}E_2^\theta}\le \|x\|_{(E_1,E_2)_\theta}\le B\|x\|_{ E_1^{1-\theta}E_2^\theta},\quad \forall x\in  E_1^{1-\theta}E_2^\theta,
$$
where $A$ and $B$  are positive constants which does not depend  on $\alpha$.
\end{theorem}
From  \cite[Proposition 7.8]{KMM}, we know that  quasi Banach lattice $X$  is analytically convex if and only if
 there exists $r > 0$ such that $X$ is $r$-convex. Hence, Theorem \ref{thm:interpolation-colderon} is a special case of  Theorem 3.4 from \cite{KM}.

\begin{remark}
Similar to Remark of  Theorem 4.1.14 of \cite{KPS}, we have that if at least one of $E_1$ and $E_2$  has order continuous norm, then $E_1^{1-\theta}E_2^\theta$ has order continuous norm.
\end{remark}

\begin{corollary}\label{cor:interpolation-n}
Let $E_j$ be a symmetric quasi Banach  function space on $(0,\alpha)$ which is
$s_j$-convex for some $0 < s_j < \8$, $j=1,\; 2$. Suppose $E=E_1^{1-\theta} E_2^{\theta}\;(0<\theta<1)$. If $E_j$ has order continuous norm, $j=1,\; 2$, then $E^{(n)}=(E_1^{(n)},E_2^{(n)})_\theta$ for any  $n\in\mathbb{N}$.
\end{corollary}

\begin{proposition}\label{pro:noncommutative-calderon}
Let $E_{1},\;E_{2}$ be  symmetric quasi Banach function spaces on $[0,1]$. Suppose $0<\theta<1$ and $E=E_1^{1-\theta}E_2^\theta$. Define
$$
E_1(\M)^{1-\theta}E_2(\M)^\theta=\{x:\;|x|\le\lambda|x_{1}|^{1-\theta}|x_{2}|^{\theta},\;x_{i}\in B(E_{i}(\M)),\;i=1,\;2\},
$$
$$
\|x\|_{E_1(\M)^{1-\theta}E_2(\M)^\theta}=\inf\{\lambda:
\;|x|\le\lambda|x_{1}|^{1-\theta}|x_{2}|^{\theta},\;x_{i}\in B(E_i(\M)),\;i=1,\;2\},
$$
then $E(\M)=E_1(\M)^{1-\theta}E_2(\M)^\theta$ with equivalent norms.
\end{proposition}

\begin{proof} Let $x\in E_1(\M)^{1-\theta}E_2(\M)^\theta$. Then for any $\varepsilon>0$, there are $\lambda>0$ and $x_{i}\in B(E_{i}(\M))$ ($i=1,\;2$) such that
$|x|\le\lambda|x_{1}|^{1-\theta}|x_{2}|^{\theta}$ and $\lambda<\|x\|_{E_1(\M)^{1-\theta}E_2(\M)^\theta}+\varepsilon$. Since
$$
\mu_t(x)\le \lambda\mu_{\frac{t}{2}}^{1-\theta}(x_1)\mu_{\frac{t}{2}}^{\theta}(x_2)=\lambda[D_2\mu(x_1)]^{1-\theta}(t)[D_2\mu(x_2)]^{\theta}(t)
$$
 and $\|D_2\mu(x_i)\|_{E_i}\le C$, $i=1,\;2$, where $C$ is a positive constant (see \cite[Lemma 2.2]{D}).  It follows that  $x\in E(\M)$ and $\|x\|_E\le C\|x\|_{E_1(\M)^{1-\theta}E_2(\M)^\theta}$.

To prove the converse let $x\in E(\M)$ and set $\N$ be a commutative von Neumann
subalgebra  containing the spectral
projections of $|x|$. We may also choose $\N$ so that the
restriction of $\tau$ to it is still semifinite.  $\N$ is
identifiable with $L_\8(\Omega, \mu)$ for some measure space $(\Omega,
\mu)$, where $\mu$ is the measure induced by $\tau$.  Hence, $|x|\in
 E(\mathcal{N})=E(\Omega,\mu)$. It follows that for every $\varepsilon>0$, there are
 $x_1\in B(E_1(\Omega,\mu))^+=B(E_1(\N))^+$ and $x_2\in B(E_2(\Omega,\mu))^+=B(E_2(\N))^+$ such that $|x|\le \lambda x_1^{1-\theta}x_2^{\theta}$ and
 $\|x\|_{E}+\varepsilon>\lambda$. Therefore, $\|x\|_E\ge\|x\|_{E_1(\M)^{1-\theta}E_2(\M)^\theta}$.  This completes the proof.
\end{proof}

Similar to Proposition \ref{pro:noncommutative-calderon}, we have the following (see \cite[Theorem 3 and 4]{S1} and \cite[Theorem 2.5] {Be}) :
\begin{proposition}\label{pro:noncommutative-ponitwise product}
 Let $E_{1},\;E_{2}$ be
 symmetric quasi Banach function spaces on $(0,\alpha)$ and  $E=E_1\odot E_2$. Set
 $$
E_1(\M)\odot E_2(\M)=\{x:\;x=x_{1}x_{2},\;x_{i}\in E_{i}(\M),\;i=1,\;2\},
$$
$$
\|x\|_{E_1(\M)\odot E_2(\M)}=\inf\{\|x_{1}\|_{E_1}\|x_{2}\|_{E_2}:
\;x=x_{1}x_{2},\;x_{i}\in E_i(\M),\;i=1,\;2\}.
$$
Then $E(\M)=E_1(\M)\odot E_2(\M)$ with equivalent norms.
\end{proposition}

We recall interpolation of noncommutative symmetric spaces. Let $E_1,\;E_2$ be fully symmetric Banach function spaces on $(0,\alpha)$  and $0<\theta<1$. If $E$ is complex interpolation  of $E_1$ and $E_2$, i.e. $E=(E_1,E_2)_\theta$. Then
\begin{equation}\label{eq:interpolation-product}
E(\M)=(E_1(\M),E_2(\M))_\theta.
\end{equation}
For more details on interpolation of noncommutative symmetric spaces we refer to
\cite{DDP2}.

\begin{lemma}\label{lem:analytic convex} Let $E$ be a symmetric quasi Banach function space on $(0,\alpha)$ equipped with  Fatou
quasi norm. If $E$ is  analytically convex, then $E(\M)$ can be equipped with
an equivalent  quasi norm $\|\cdot\|$ so that $\|\cdot\|$    is
plurisubharmonic.  Consequently, $E(\M)$ is  analytically convex.
\end{lemma}
\begin{proof} If $(E,\;\|\cdot\|_E)$ is analytically convex, then there exists $r > 0$ such that $E$ is $r$-convex (see \cite[Proposition 7.8]{KMM}).  $E$  can be given an equivalent quasi norm $\|\cdot\|'_E$ so that $(E,\;\|\cdot\|'_E)$ is a symmetric quasi Banach function space whose
 $r$-convexity constant is equal to 1 (see \cite[Corollary 3.5]{DDS}). By Lemma 4.2 from \cite{X}, $(E(\M),\;\|\cdot\|'_E)$  is
plurisubharmonic. Hence $E(\M)$ is analytically convex (see \cite{K}). This completes the proof.
\end{proof}

For $x\in\M$, we denote by $r(x)$ the right support of $x$, i.e., $r(x)$ is the smallest projection $e$ in  $\M$ such that $xe=x$. Set
$F(\M)=\{x:\; x\in\M,\; \tau(r(x))<\8\}$. Let $A(S)$ be the space of complex valued functions, analytic in $S$ and continuous and bounded
in  $\overline{S}$.
Denote $AF(\M)$ the family of functions of the form $f(z)=\sum_{k=1}^n f_k(z)x_k$ with $f_k$ in $A(S)$ and $x_k$ in $F(\M)$.

Let $E_j$ be a separable symmetric quasi Banach function space on $(0,\alpha)$ such that
$s_j$-convex for some $0 < s_j < \8$, $j=1,\; 2$.  It is clear that $E_1(\M)+E_2(\M)=(E_1+E_2)(\M)$ in the sense of equivalence of quasi norms (see Theorem \ref{K-functional for noncommutative}).  Since  $E_1+E_2$ is  separable and  $s$-convex ($s=\min\{s_1,s_2\}$), from the proof of Lemma \ref{lem:analytic convex}, we obtain that there is an equivalent quasi norm $\|\cdot \|$ on $E_1(\M)+E_2(\M)$ which is plurisubharmonic. On the other hand, since $E_1,\;E_2$ and $ E_1+E_2$ are  separable and $s$-convex, by \cite[Lemma 4.1]{X}, $F(\M)$ is dense in $E_1(\M),\;E_2(\M)$ and $E_1(\M)+E_2(\M)$, respectively. Using the same method as  in the proof of \cite[Theorem 4.4]{X}, we  deduce  that  $AF(\M)$  dense in $\F(E_{1}(\M),E_{2}(\M))$. Hence, we use \eqref{eq:interpolation-product} and Corollary \ref{cor:interpolation-n}  to obtain the following result:

\begin{theorem}\label{thm:interpolation} Let $E_j$ be a symmetric quasi Banach function space on $(0,\alpha)$ which is
$s_j$-convex for some $0 < s_j < \8$, $j=1,\; 2$. Suppose  $E_j$ has order continuous norm, $j=1,\; 2$ and $0<\theta<1$. If $E=E_1^{1-\theta}E_2^\theta$, then
\be
E(\M)=(E_1(\M),E_2(\M))_\theta
\ee
and
$$
A\|x\|_{ E}\le \|x\|_{(E_1(\M),E_2(\M))_\theta}\le B\|x\|_{E},\quad \forall x\in  E(\M),
$$
where $A$ and $B$  are positive constants which  does not depend  on $\alpha$.
\end{theorem}
\begin{proof}
 Chose $n\in\mathbb{N}$ such that $ns_j\ge1\;(j=1,\;2)$. Then   $E_j^{(n)}$ can be renormed as a symmetric Banach function space ($j=1,\;2$).  In the following, we consider  $E_j^{(n)}$ with this new symmetric norm ($j=1,\;2$).

  Let $x\in E(M)$ and  $x=u|x|$ be the  polar decomposition of $x$. Set $x_1=u|x|^\frac{1}{n}$ and $x_k=|x|^\frac{1}{n}$ for $2\le k\le n$. Then
 $x=x_1x_2\cdots x_n$ and $x_k\in E^{(n)}(\M)=(E_1^{(n)})^{1-\theta}(E_2^{(n)})^\theta(\M)\;(1\le k\le n)$. By \eqref{eq:interpolation-product},
it follows that for any $\varepsilon>0$, there exists $f_k\in\mathcal{F}(E_{1}^{(n)}(\M),E_{2}^{(n)}(\M))$ such that $f_k(\theta)=x_k$ and
$$\begin{array}{rl}
\big\|f_k\big\|_{\F(E_{1}^{(n)},E_{1}^{(n)})}&=\max\{\sup_{t\in\mathbb{R}}\|f_k(it)\|_{E_1^{(n)}},\;\sup_{t\in\mathbb{R}}\|f_k(1+it)\|_{E_2^{(n)}}\}\\
&<\|x_k\|_{E^{(n)}}+\varepsilon.
\end{array}
$$
Set $f=\prod_{k=1}^nf_k$. Then  $f\in\F(E_{1}(\M),E_{2}(\M))$ and $f(\theta)=x$. By Proposition \ref{pro:noncommutative-ponitwise product},
$$\begin{array}{rl}
    \|x\|_{(E_1(\M),E_2(\M))_\theta}&\le\big\|f\big\|_{\F(E_{1}(\M),E_{2}(\M))}=\max\{\sup_{t\in\mathbb{R}}\|f(it)\|_{E_1},\;\sup_{t\in\mathbb{R}}\|f(1+it)\|_{E_2}\}\\
    &=\max\{\sup_{t\in\mathbb{R}}\|\prod_{k=1}^nf_k(it)\|_{E_1},\;\sup_{t\in\mathbb{R}}\|\prod_{k=1}^nf_k(1+it)\|_{E_2}\}\\
    &\le B\max\{\sup_{t\in\mathbb{R}}\prod_{k=1}^n\|f_k(it)\|_{E_1^{(n)}},\;\sup_{t\in\mathbb{R}}\prod_{k=1}^n\|f_k(1+it)\|_{E_2^{(n)}}\}\\
    &\le B\max\{\prod_{k=1}^n\sup_{t\in\mathbb{R}}\|f_k(it)\|_{E_1^{(n)}},\;\prod_{k=1}^n\sup_{t\in\mathbb{R}}\|f_k(1+it)\|_{E_2^{(n)}}\}\\
    &\le B\prod_{k=1}^n\max\{\sup_{t\in\mathbb{R}}\|f_k(it)\|_{E_1^{(n)}},\;\sup_{t\in\mathbb{R}}\|f_k(1+it)\|_{E_2^{(n)}}\}\\
    &<B\prod_{k=1}^n(\|x_k\|_{E^{(n)}}+\varepsilon).
  \end{array}
$$
Letting $\varepsilon\rightarrow0$, we get
$$
\|x\|_{(E_1(\M),E_2(\M))_\theta}\le B\prod_{k=1}^n\|x_k\|_{E^{(n)}}=B\|x\|_{E}.
$$

Conversely, it is enough to prove that if $x\in F(\M)$ and $ f\in AF(\M)$ with $f(\theta)=x$, then
\beq\label{eq:product-inter2}
\|x\|_{E}\le C\max\big\{
 \sup_{t\in\real}\big\|f(it)\big\|_{E_1}\,,\;
 \sup_{t\in\real}\big\|f(1+it)\big\|_{E_2}\big\}.
\eeq
Using reversed H$\rm\ddot{o}$lder's inequality (see \cite[lemma 2.2]{Xu1}), we obtain that
for any $\varepsilon>0$, there exist $f_1,\;f_2,\cdots,\;f_n\in AF(\M)$ such that $f=\prod_{k=1}^nf_k$ and for all $z\in\partial S,\;t>0$,
$$
\mu_t(f_k(z))<\mu_t(f(z))^\frac{1}{n}+\varepsilon,\quad k=1,2,\cdots,n.
$$
Hence,
$$
\begin{array}{l}
\max\big\{\sup_{t\in\real}\big\|f_k(it)\big\|_{E_1^{(n)}}\,,\; \sup_{t\in\real}\big\|f_k(1+it)\big\|_{E_2^{(n)}}\big\}\\
\qquad\le\max\big\{\sup_{t\in\real}\big\|\mu_t(f(it))^\frac{1}{n}\big\|_{E_1^{(n)}}\,,\; \sup_{t\in\real}\big\|\mu_t(f(1+it))^\frac{1}{n}\big\|_{E_2^{(n)}}\big\}+O(\varepsilon)\\
\qquad=\max\big\{ \sup_{t\in\real}\big\|f(it)\big\|_{E_1}^\frac{1}{n}\,,\; \sup_{t\in\real}\big\|f(1+it)\big\|_{E_2}^\frac{1}{n}\big\}+O(\varepsilon),\quad k=1,2,\cdots,n.
\end{array}
$$
 Set $x_k=f_k(\theta)$. Then $x=\prod_{k=1}^nx_k$. By \eqref{eq:interpolation-product},
$$
x_k\in (E_1^{(n)}(\M),E_2^{(n)}(\M))_\theta=(E_1^{(n)})^{1-\theta}(E_2^{(n)})^\theta(\M)=E^{(n)}(\M)
$$
and
$$
\|x_k\|_{E^{(n)}}\le \max\{\sup_{t\in\mathbb{R}}\|f_k(it)\|_{E_1^{(n)}},\;\sup_{t\in\mathbb{R}}\|f_k(1+it)\|_{E_2^{(n)}}\},\quad 1\le k\le n.
$$
Hence, by Proposition \ref{pro:noncommutative-ponitwise product},
$$\begin{array}{rl}
\|x\|_{E}&\le C\prod_{k=1}^n\|x_k\|_{E^{(n)}}\\
&\le C\prod_{k=1}^n\max\{\sup_{t\in\mathbb{R}}\|f_k(it)\|_{E_1^{(n)}},\;\sup_{t\in\mathbb{R}}\|f_k(1+it)\|_{E_2^{(n)}}\}\\
&\le C\prod_{k=1}^n\max\big\{ \sup_{t\in\real}\big\|f(it)\big\|_{E_1}^\frac{1}{n}\,,\; \sup_{t\in\real}\big\|f(1+it)\big\|_{E_2}^\frac{1}{n}\big\}+O(\varepsilon)\\
&\le C\max\big\{
 \sup_{t\in\real}\big\|f(it)\big\|_{E_1}\,,\;
 \sup_{t\in\real}\big\|f(1+it)\big\|_{E_2}\big\}+O(\varepsilon).
\end{array}
$$
From this follows \eqref{eq:product-inter2}. This finishes the proof.

\end{proof}

 Using Theorem  \ref{thm:interpolation}, Proposition \ref{pro:noncommutative-calderon} and \ref{pro:noncommutative-ponitwise product}, we get the following result.
\begin{corollary}
Let $E_j$ be a symmetric quasi Banach  function space on $(0,\alpha)$ which is
$s_j$-convex for some $0 < s_j < \8$, $j=1,\; 2$. Suppose $E_j$ has order continuous norm $(j=1,\;2)$ and $0<\theta<1$. If $E=E_1^{1-\theta}E_2^\theta$, then
\be
(E_1(\M),E_2(\M))_\theta=E_1^{1-\theta}(\M)E_2^{\theta}(\M)=E_1^{(\frac{1}{1-\theta})}(\M)\odot E_2^{(\frac{1}{\theta})}(\M)
\ee
with equivalent norms.
\end{corollary}


\section{Complex interpolation of noncommutative symmetric Hardy spaces}

We will assume that $\mathcal{D}$ is a von Neumann
 subalgebra of $\mathcal{M}$ such that the restriction of $\tau$ to $\mathcal{D}$ is still semifinite. Let $\mathcal{E}$ be the (unique) normal faithful
 conditional expectation of $\mathcal{M}$ with respect to $\mathcal{D}$ which
 leaves $\tau$ invariant.

\begin{definition} A w*-closed subalgebra $\A$ of $\M$ is
called a subdiagonal algebra of $\M$ with respect to $\E$(or
$\D$) if \begin{enumerate}[\rm (i)]

\item $\A+ J(A)$ is w*-dense in  $\M$, where $J(A)=\{x^*:\;x\in\A\}$,

\item $\E(xy)=\E(x)\E(y),\; \forall\;x,y\in
\A$,

\item $\A\cap J(A)=\D$.
\end{enumerate}
$\D$ is then called the diagonal of $\A$.
\end{definition}

 Recall that a
subdiagonal algebra $\mathcal{A}$ is automatically  maximal in the
sense that if $\mathcal{B}$ is another subdiagonal algebra with
respect to $\mathcal{E}$ containing $\mathcal{A}$, then
$\mathcal{B}=\mathcal{A}$ (see \cite{E,J}).

\begin{definition}\label{symetrichardysemifinite-definition}
 Let  $E$ be a symmetric (quasi) Banach function space on
$(0,\alpha)$ and $\A$  a semifinite subdiagonal algebra of $\M$. Then
$$
E(\A)=\mbox{closure
 of}\;\mathcal{A}\cap E(\M)\; \mbox{in
 }\; E(\M)
 $$
 is called symmetric (quasi)
Hardy space associated with $\A$.
\end{definition}
If $\M$ is finite ($\tau(1)=\alpha<\8$) and  $E$ is a separable symmetric $s$-convex quasi Banach  space on
$(0,\alpha)$  for some $0 < s < \8$.
By Lemma 4.4 in \cite{Be}, we have that
  \begin{equation}\label{saito}
    E(\A)=H_s(\A)\cap E(\M).
 \end{equation}

\subsection{Finite case}
In this subsection  $\M$ always denotes a finite  von Neumann algebra  with a normal  faithful trace $\tau$ satisfying $\tau(1)=\alpha<\8$. We keep all notations introduced in the last section.

\begin{theorem}\label{interpolationhardy-finite}
Let $E_j$ be a  quasi Banach  function space on $(0,\alpha)$ which is
$s_j$-convex for some $0 < s_j < \8$, $j=1,\;2$ and $0<\theta<1$. Suppose $E_j$ has order continuous norm, $j=1,\;2$, and $E=E_1^{1-\theta}E_2^\theta$. If
 $q_{E_j}<\8$, $j=1,\;2$,
then
$$
(E_1(\A),E_2(\A))_\theta=E(\A)
$$
and
$$
A\|x\|_{ E}\le \|x\|_{(E_1(\A),E_2(\A))_\theta}\le B\|x\|_{E},\quad \forall x\in  E(\A),
$$
where $A$ and $B$ are positive constants and  do not depend  on $\alpha$.
\end{theorem}
\begin{proof}
Using Theorem \ref{thm:interpolation}, we deduce that $(E_1(\M),E_2(\M))_\theta=E(\M)$. Hence,
$$
(E_1(\A),E_2(\A))_\theta\subset(E_1(\M),E_2(\M))_\theta=E(\M).
$$
 Using
 Lemma 4.5 in \cite{D}, we deduce that
 $E_1(\A),E_2(\A) \subset H_s(\A)$, where $0<s<\min\{s_1,s_2\}$. This implies
 $$
 (E_1(\A),E_2(\A))_\theta\subset H_s(\A).
 $$
Since $E$  has order continuous norm, by
 \eqref{saito}, we obtain that
\beq\label{eq:interpolation-inclusion}
(E_1(\A),E_2(\A))_\theta\subset E(\A).
\eeq
Hence, by Theorem \ref{thm:interpolation},
\beq\label{isometry1}
\|a\|_E\le \frac{1}{A}\|a\|_{(E_1(\A),E_2(\A))_\theta},\quad \forall a\in(E_1(\A),E_2(\A))_\theta.
\eeq
 Chose $n\in\mathbb{N}$ such that $n\min\{s_1,s_2\}>1$. By Corollary \ref{cor:interpolation-n}, we have  $E^{(n)}=(E_1^{(n)},E_2^{(n)})_\theta$. Since $1< p_{E_j^{(n)}}\le q_{E_j^{(n)}} <\8\;(j=1,\;2)$, using Theorem 6 in \cite{BM} and \eqref{saito}, we get
\beq\label{eq:interpolation2}
(E_1^{(n)}(\A),E_2^{(n)}(\A))_\theta=E^{(n)}(\A).
\eeq
 It is clear that $E=E^{(n)}\odot E^{(n)}\odot\cdots \odot E^{(n)}$.  Consequently, from  the proof of Theorem 4 in \cite{BM} and \eqref{saito} it follows that
\beq\label{eq:product2}
E(\A)=E^{(n)}(\A)\odot E^{(n)}(\A)\odot\cdots \odot E^{(n)}(\A).
 \eeq
Let $x\in E(\A)$. By \eqref{eq:product2}, for $\varepsilon>0$, there exist $x_k\in E^{(n)}(\A)\;(k=1,2,\cdots,n)$ such that $x=\prod_{k=1}^nx_k$ and $\prod_{k=1}^n\|x_k\|_{E^{(n)}}<\|x\|_{E}+\varepsilon$. By \eqref{eq:interpolation2}, there exist $F_k\in\mathcal{F}(E_{1}^{(n)}(\A),E_{2}^{(n)}(\A))$ such that $F_k(\theta)=x_k$ and
$$
\|F_k\|_{\mathcal{F}(E_{1}^{(n)}(\A),E_{2}^{(n)}(\A))}=\max\{\sup_{t\in\mathbb{R}}\|F_k(it)\|_{E_{1}^{(n)}},\;\sup_{t\in\mathbb{R}}\|F_k(1+it)\|_{E_{2}^{(n)}}\}<C\|x_k\|_{E^{(n)}}+\varepsilon
$$
$( k=1, 2\cdots,n)$, where $C$ is constant.
Put $F=\prod_{k=1}^nF_k$. Then
\beq\label{eq:product-inclusion}
F\in\mathcal{F}(E_{1}(\A),E_{2}(\A)), \qquad F(\theta)=x.
\eeq
By Proposition \ref{pro:noncommutative-ponitwise product},
\be\begin{array}{rl}
     \|F\|_{\mathcal{F}(E_{1}(\A),E_{2}(\A))} &=\max\{\sup_{t\in\mathbb{R}}\|F(it)\|_{E_{1}},\;\sup_{t\in\mathbb{R}}\|F(1+it)\|_{E_{2}}\}  \\
      & \le C'\max\{\sup_{t\in\mathbb{R}}\prod_{k=1}^n\|F_k(it)\|_{E_1^{(n)}},\prod_{k=1}^n\|F_k(1+it)\|_{E_2^{(n)}}\}\\
      &\le C'\max\{\sup_{t\in\mathbb{R}}\prod_{k=1}^n\|F_k(it)\|_{E_1^{(n)}},\;\sup_{t\in\mathbb{R}}\prod_{k=1}^n\|F_k(1+it)\|_{E_2^{(n)}}\}\\
      &<C'\prod_{k=1}^n(C\|x\|_{E^{(n)}}+\varepsilon).
   \end{array}
\ee
Hence,
$$
 \|x\|_{(E_1(\A),E_2(\A))_\theta}<C'\prod_{k=1}^n(C\|x\|_{E^{(n)}}+\varepsilon).
 $$
 Letting $\varepsilon\rightarrow0$, we obtain that $ \|x\|_{(E_1(\A),E_2(\A))_\theta} \le C'C^n\|x\|_{E}$. Hence,
\be
E(\A)\subset(E_1(\A),E_2(\A))_\theta.
\ee
This completes the proof.
\end{proof}

Using (ii) of Proposition \ref{thm:product-calderon} and Theorem \ref{thm:interpolation}, the same way as in in \cite[Theorem 4]{BM}, we obtain the following result.

\begin{corollary}\label{cor:product-hardy} Let $E_j$ be a symmetric quasi Banach function space on $(0,\alpha)$ with order continuous norm which is
$s_j$-convex for some $0 < s_j < \8$, $j=1,\;2$ and $0<\theta<1$.  If
 $q_{E_1}, q_{E_2}<\8$,
then
$$
(E_1\A),E_2(\A))_\theta=E_1^{(\frac{1}{1-\theta})}(\A)\odot E_2^{(\frac{1}{\theta})}(\A).
$$
\end{corollary}

\subsection{Semifinite case}
In this subsection  $\M$ always denotes a semifinite von Neumann algebra  with a normal semi-finite faithful trace $\tau$ satisfying $\tau(1)=\8$. We keep all notations introduced in the third section.

Given a projection  $e$ in $\mathcal{D}$, we let
\be
\mathcal{M}_{e}=e\mathcal{M}e,\quad \mathcal{A}_{e}=e\mathcal{A}e,\quad
\mathcal{D}_{e}=e\mathcal{D}e,
\ee
and  $\mathcal{E}_{e}$ be the restriction of $\E$  to $\mathcal{M}_{e}$.
 Then
  $\mathcal{A}_{e}$ is a subdiagonal algebra of
 $\mathcal{M}_{e}$ with respect to $\mathcal{E}_{e}$ and with diagonal
 $\mathcal{D}_{e}$.

Since  $\mathcal{D}$ is semifinite, we can choose an increasing family of
$\{e_{i}\}_{i\in I}$ of $\tau$-finite projections in $\mathcal{D}$ such that
$e_{i}\rightarrow1$  strongly, where 1 is identity of $\mathcal{M}$ (see Theorem 2.5.6 in \cite{Sa}). Throughout, the $\{e_{i}\}_{i\in I}$ will be used to indicate this net.

\begin{lemma}\label{convergence}
Let $E$ be  a  symmetric quasi Banach  function space on $(0,\8)$ which is
$s$-convex for some $0 < s < \8$. Suppose that  $0<p<  p_E\le q_E<q<\8$.
If  $(a_i)_{i\in I}$ is a bounded net in $\M$ converging strongly to $a$, then $xa_i\rightarrow xa$ in the norm of
$E(\M)$ for any $x\in L_p(\M)\cap \M$.
\end{lemma}
\begin{proof}  By Lemma 4.5 in \cite{D}, it follows that
\beq\label{inclutionp-q}
L_p(\M)\cap L_q(\M)\subset E(\M)\subset L_p(\M)+ L_q(\M)
\eeq
with continuous inclusions.
On the other hand, if $x\in  L_p(\M)\cap \M\subset L_p(\M)\cap L_q(\M)$, by Lemma 2.3 in \cite{Ju}, we have that $xa_i\rightarrow xa$ in the norm of
$L_p(\M)$  and $xa_i\rightarrow xa$ in the norm of
$L_q(\M)$. Using \eqref{inclutionp-q} we obtain the desired result.

\end{proof}

\begin{lemma}\label{semi-closure}
Let $E$ be  a  symmetric quasi Banach function space on $(0,\8)$ which is
$s$-convex for some $0 < s < \8$. If  $E$  has order continuous norm and $ q_E <\8$, then
$$
\lim_{i}\|e_ixe_i-x\|_E=0,\qquad \forall x\in E(\M).
$$
\end{lemma}
\begin{proof} Since $E^{(\frac{1}{s})}$ can be renormed as a symmetric Banach function space, if we chose $n\in\mathbb{N}$ such that $ns>1$, then  $E^{(n)}$ is a symmetric Banach function space with order continuous norm. By Lemma 4.5 in \cite{X}, $L_1(\M)\cap\M$ is norm dense in $E^{(n)}(\M)$, so we get  $L_{\frac{1}{n}}(\M)\cap\M$ is norm dense in $E(\M)$. Let $x\in E(\M)$. Then for any $\varepsilon>0$, there
exists $y\in L_\frac{1}{n}(\M)\cap\M$ such that
$\|x-y\|_{E(\M)}<\frac{\varepsilon}{3K^2}$, where $K$ is quasi norm constant of $\|\cdot\|_E$. It is clear that $\frac{1}{n}<s\le  p_E$, by Lemma \ref{convergence}, we get
$\lim_{i}\|ye_{i}- y\|_{E}=0$ and $\lim_{i}\|e_{i}y-y\|_{E}=0$. Hence
$$
\lim_{i}\|y-e_{i}ye_{i}\|_{E}\leq K[\lim_{i}\|y-ye_{i}\|_{E}+\lim_{i}\|(y-e_{i}y)e_{i}\|_{E}]=0.
$$
Thus there exists $i_{0}\in I$ such that
$$\|y-e_{i_{0}}ye_{i_{0}}\|_{E(\M}<\frac{\varepsilon}{3K^2},\qquad i\ge i_0.$$
Hence, for $i\ge i_0$,
$$
\|x-e_{i_{0}}xe_{i_{0}}\|_{E}\le K^2[\|x-y\|_{E}+\|e_{i_{0}}xe_{i_{0}}-e_{i_{0}}y e_{i_{0}}\|_{E}+\|y-e_{i_{0}}ye_{i_{0}}\|_{E}]<\varepsilon.
$$
Thus $\lim_{i}e_{i}xe_{i}=x$.
\end{proof}

\begin{lemma}\label{lem:saito-type} Let $E$ be a  symmetric quasi Banach function space on $(0,\8)$ which is
$s$-convex for some $0 < s < \8$. Suppose that  $0<p<  p_E\le q_E<q<\8$. If $E$ has order continuous norm, then
\beq\label{eq:saito-type}
E(\A)=(H_p(\A)+H_q(\A))\cap E(\M).
\eeq
\end{lemma}
\begin{proof} We claim that
\beq\label{eq:inclustion-semi-hardy}
E(\A)\subseteq H_p(\A)+H_q(\A).
\eeq
Indeed, if $x\in E(\A)$, then there exists a sequence $(x_n)\subset\A\cap E(\M)$ such that $\|x-x_n\|_E\rightarrow0$ as $n\rightarrow\8$.
By \eqref{inclutionp-q}, $\|x-x_n\|_{L_p(\M)+L_q(\M)}\rightarrow0$ as $n\rightarrow\8$. On the other hand, $x_n\in\A\cap E(\M)\subset\A\cap(L_p(\M)+L_q(\M))$ for $n\in \mathbb{N}$. Hence,
$$
e_ix_ne_i\in\A_{e_i}\cap(L_p(\M_{e_i})+L_q(\M_{e_i}))\subset H_p(\A_{e_i})\subset H_p(\A)+H_q(\A),\quad \forall i\in I.
$$
Since $\lim_{i}\|x_n-e_ix_ne_i\|_{L_p(\M)+L_q(\M)}=0$, we get $x_n\in H_p(\A)+H_q(\A)$ for $n\in \mathbb{N}$. Therefore, $x\in H_p(\A)+H_q(\A)$, which gives \eqref{eq:inclustion-semi-hardy}. Thus $E(\A)\subset(H_p(\A)+H_q(\A))\cap E(\M)$.

Conversely, if $y\in (H_p(\A)+H_q(\A))\cap E(\M)$,
then
$$
e_iye_i\in (H_p(\A_{e_i})+H_q(\A_{e_i}))\cap E(\M_{e_i})\subset H_p(\A_{e_i})\cap E(\M_{e_i}),\quad \forall i\in I.
$$
By \eqref{saito}, we have that
$$
e_iye_i\in E(\A_{e_i})\subset E(\A),\quad \forall i\in I.
$$
On the other hand, since $E$ has order continuous norm, by Lemma \ref{semi-closure}, we have that
$\lim_{i}\|e_iye_i- y\|_{E}=0$.
Thus $y\in E(\A)$. This gives the desired result.
\end{proof}

\begin{theorem}\label{interpolationhardy-semi-finite}
Let $E_j$ be a symmetric quasi Banach function space on $(0,\alpha)$ with order continuous norm which is
$s_j$-convex for some $0 < s_j < \8$ $(j=1,\;2)$. Suppose $0<\theta<1$ and $E=E_1^{1-\theta}E_2^\theta$. If
 $q_{E_1},\; q_{E_2}<\8$, then
$$
(E_1(\A),E_2(\A))_\theta=E(\A)
$$
with equivalent norms.
\end{theorem}
\begin{proof}
 By \eqref{eq:interpolation-product}, we deduce that
$$
(E_1(\A),E_2(\A))_\theta\subset(E_1(\M),E_2(\M))_\theta=E(\M).
$$
Let $0<p< \min\{p_{E_1},p_{E_2}\}\le \max\{q_{E_1},q_{E_2}\}<q<\8$. Then by \eqref{eq:inclustion-semi-hardy},
$$
E_j(\A)\subset H_p(\A)+ H_q(\A),\qquad j=1,\;2.
$$
Hence,
$$
(E_1(\A),E_2(\A))_\theta\subset E_1(\A)+E_2(\A) \subset H_p(\A)+ H_q(\A).
$$
By Lemma \ref{lem:saito-type},
$$
(E_1(\A),E_2(\A))_\theta\subset (H_p(\A)+ H_q(\A))\cap E(\M)=E(\A).
$$
Hence, by Theorem \ref{thm:interpolation}
\beq\label{isometry-semi-1}
\|a\|_{E}\le\frac{1}{A}\|a\|_{(E_1(\A),E_2(\A))_\theta},\quad \forall a\in(E_1(\A),E_2(\A))_\theta.
\eeq
Using Theorem \ref{interpolationhardy-finite}, we get that
 $$
(E_1(\A_{e_i}),E_2(\A_{e_i}))_\theta=E(\A_{e_i}),\qquad \forall i\in I.
 $$
Hence,
 $$
E(\A_{e_i})=(E_1(\A_{e_i}),E_2(\A_{e_i}))_\theta\subset (E_1(\A),E_2(\A))_\theta
 $$
 and
 \beq\label{inequality-semi}
  \|a\|_{(E_1(\A),E_2(\A))_\theta}\le B\|a\|_{E},\qquad \forall a\in E(\A_{e_i}),
 \eeq
 where $B$ does not depends to $i\in I$.
  Let $a\in E(\A)$. By Lemma \ref{semi-closure},
 $$
\lim_{i}\| a -e_iae_i\|_{E}=0.
 $$
Since $\{e_{i}\}_{i\in I}$ is increasing and $I$ is  a directed set, using \eqref{inequality-semi} we obtain that
$$
\|e_iae_i-e_jae_j\|_{(E_1(\A),E_2(\A))_\theta}\le B\|e_iae_i-e_jae_j\|_{E},\qquad  \forall i,\;j\in I.
$$
Hence $\{e_iae_i\}_{i\in I}$ is a Cauchy net  in $(E_1(\A),E_2(\A))_\theta$.
 So, there exists  $y$ in $(E_1(\A),E_2(\A))_\theta$ such that
$$
\lim_{i}\|e_iae_i-y\|_{(E_1(\A),E_2(\A))_\theta}=0.
$$
By \eqref{isometry-semi-1}, $\lim_{i}\|e_iae_i-y\|_{E}=0$. Therefore, $y=a$. We  deduce that
$$
  \|a\|_{(E_1(\A),E_2(\A))_\theta}\le B\|a\|_{E},\quad \forall a\in E(\A).
$$
Thus $(E_1(\A),E_2(\A))_\theta=E(\A)$, and so the proof is complete.
\end{proof}

Let $E_j$ be a symmetric quasi Banach function space on $(0,\alpha)$ with  $0 < p_{E_j}\le q_{E_j} < \8$, $j=1,\;2$.  We choose $p,\;q$ such that $0<p< \min\{p_{E_1},p_{E_2}\}\le \max\{q_{E_1},q_{E_2}\}<q<\8$.
Set
\beq\label{eq:hardy-cordalon}
E_1(\A)^{1-\theta}E_2^{\theta}(\A)=E_1(\M)^{1-\theta}E_2^{\theta}(\M)\cap(H_p(\A)+H_q(\A)),\quad  0<\theta<1.
\eeq
By Lemma \ref{lem:saito-type}, Proposition \ref{pro:noncommutative-calderon} and Theorem \ref{interpolationhardy-semi-finite}, we have that

\begin{corollary}
Let $E_j$ be a symmetric quasi Banach  function space on $(0,\alpha)$ with order continuous norm which is
$s_j$-convex for some $0 < s_j < \8$, $j=1,\;2$.  If
 $q_{E_1}, q_{E_2}<\8$, then
$$
(E_1(\A),E_2(\A))_\theta=E_1(\A)^{1-\theta}E_2^{\theta}(\A), \quad 0<\theta<1,
$$
with equivalent norms.
\end{corollary}

\section{Real interpolation of noncommutative symmetric Hardy spaces}

\begin{theorem}\label{K-functional for noncommutative} Let $E_j$ be a symmetric quasi Banach function space on $(0,\alpha)$ which is
$s_j$-convex for some $0 < s_j < \8,\; j=1,\;2$. Then there exist constants $A>0,\; B>0$ such that  for all $x\in E_1(\M)+E_2(\M)$
and  all $t>0$
 \beq\label{eq:K-function-equivalent}
 AK_t\big(x, E_1(\M),E_2(\M)\big)
\le K_t\big(\mu(x),E_1,E_2\big)\le BK_t\big(x, E_1(\M),E_2(\M)\big),
\eeq
where $A,\; B$  depend only on $E_1$ and $E_2$.
 \end{theorem}

\begin{proof} First suppose that $p_{E_j}>1\;(j=1,\;2)$. Let $x\in E_1(\M)+E_2(\M)$. By Lemma 4.5 in \cite{D}, $x\in L_1(\M)+L_\8(\M)$. Hence, by Theorem 2.1 in \cite{PX},
there exists a linear map  $T:\;L_1(\M)+L_\infty(\M) \to L_1(0,\alpha)+L_\infty(0,\alpha)$ (resp. $S: L_1(0,\alpha)+L_\infty(0,\alpha)\to
L_1(\M)+L_\infty(\M)$ ) such that
$T$ $($resp. $S)$ is a contraction from $L_p(\M)$
into $L_p(0,\alpha)$ (resp. from $L_p(0,\alpha)$ into $L_p(\M)$) for $p=1$ and $p=\infty$, and $Tx=\mu(x)$ (resp. $S\mu(x)=x$).  Using \cite[Theorem 4.8]{D}, we get that $T$ (resp. $S$) is a bounded map from $E_j(\M)$ to $E_j$ (resp. from $E_j$ to $E_j(\M)$) ($j=1,2$).
\be
\begin{array}{l}
     K_t\big(x, E_1(\M),E_2(\M)\big)\\
     \qquad =K_t\big(S\mu(x), E_1(\M),E_2(\M)\big)\\
     \qquad=\inf\left\{\|x_1\|_{E_1}+t\|x_2\|_{E_2}:\; S\mu(x)=x_1+x_2,\;x_j\in E_j(\M),\;j=1,\;2 \right\}  \\
      \qquad \leq \inf\left\{\|Sf_1\|_{E_1}+t\|Sf_2\|_{E_2}:\; \mu(x)=f_1+f_2,\;f_j\in E_j,\;j=1,\;2 \right\}\\
      \qquad \leq C\inf\left\{\|f_1\|_{E_1}+t\|f_2\|_{E_2}:\; \mu(x)=f_1+f_2,\;f_j\in E_j,\;j=1,\;2 \right\}\\
     \qquad =C K_t\big(\mu(x),E_1,E_2\big).
   \end{array}
\ee
Similarly, we use $T$ to prove the right inequality of \eqref{eq:K-function-equivalent}.

Now suppose that $1\ge p_{E_j}>\frac{1}{2}\;(j=1,\;2)$. Using the polar decomposition of $x$ we obtain that
 \be
 K_t\big(x, E_1(\M),E_2(\M)\big)=
 K_t\big(|x|, E_1(\M),E_2(\M)\big).
 \ee
 Hence, we assume that $x$ is positive. Let $x=x_1 +x_2$ with $x_j\in
E_j(\M)$ ($j=1,\; 2$). Then for $s>0$
 \be
 \mu_s(x)\le \mu_{s/2}(x_1)+\mu_{s/2}(x_2)\;{\mathop=^{\rm
 def}}\; f_1+f_2.
 \ee
By \cite[Lemma 2.2]{D}, we have
 $$\|f_j\|_{E_j}\le C\,\|x_j\|_{E_j}\,,\quad j=1,\;2.$$
Hence
 \be
 K_t\big(\mu(x),E_1,E_2\big)
 &\le& \|f_1\|_{E_1}+t\,\|f_2\|_{E_2}\\
 &\le&C\,\big(\|x_1\|_{E_1}+t\,
 \|x_2\|_{E_2}\big),
 \ee
so that
\be
K_t\big(\mu(x),E_1,E_2\big)
 \le C\,
  K_t\big(x, E_1(\M),E_2(\M)\big).
\ee
Conversely,
 \be\begin{array}{l}
 \big[K_t\big(\mu(x),E_1,E_2\big)\big]^{1/2}\\
 \qquad=[\inf\left\{\|f_1\|_{E_1}+t\|f_2\|_{E_2}:\; \mu(x)=f_1+f_2,\;f_j\in E_j^+,\;j=1,\;2 \right\}]^{1/2} \\
 \qquad\ge 2^{-1/2}\inf\left\{\|f_1\|_{E_1}^{1/2} +t^{1/2} \|f_2\|_{E_2}^{1/2} :\; \mu(x)^{1/2}\le f_1^{1/2} +f_2^{1/2},\;f_j \in E_j^+,\;j=1,\;2 \right\}\\
 \qquad\ge 2^{-1/2}\,K_{t^{1/2}}\big(\mu(x)^{1/2},E_1^{(2)},E_2^{(2)}\big)\\
 \qquad=2^{-1/2}\,K_{t^{1/2}}\big(\mu(x^{1/2}),E_1^{(2)},E_2^{(2)}\big).
 \end{array}
 \ee
Since $p_{E_j^{(2)}}>1\;(j=1,\;2)$, we can use the previous case to
obtain that there is a decomposition $x^{1/2}=y_1+y_2$ with $y_j\in
E_j^{(2)}(\M)^+$ ($j=1,\; 2$) such that
 \be
 \|y_1\|_{E_1^{(2)}}+t^{1/2}\,\|y_2\|_{E_1^{(2)}}\le C
K_{t^{1/2}}\big(\mu(x^{1/2}),E_1^{(2)},E_2^{(2)}\big)
 +\e
 \ee
($\e>0$ being arbitrarily given). By the operator convexity of the
map $y\mapsto y^2$, we then have
 $$x=(y_1+y_2)^2\le 2(y_1^2+ y_2^2)\;{\mathop=^{\rm def}}\;
 x_1+x_2.$$
Thus
 \be
 K_t\big(x, E_1,E_2\big)
 &\le&\|x_1\|_{E_1}+t\,\|x_2\|_{E_2}
 =2\big(\|y_1\|_{E_1^{(2)}}^{2}+t\,\|y_2\|_{E_2^{(2)}}^{2}\big)\\
 &\le& 2 \big(\|y_1\|_{E_1^{(2)}}+t^{1/2}\,\|y_2\|_{E_2^{(2)}}\big)^2.
 \ee
Putting the preceding inequalities together, we get
 \be
 K_t\big(x, E_1(\M),E_2(\M)\big)
 \le 4C^2K_t\big(\mu(x),E_1,E_2)\big).
 \ee
 If $\frac{1}{2}\ge p_{E_j}>\frac{1}{4}\;(j=1,\;2)$, then arguing as above and
using the case just proved, we can make the same conclusion as before. Finally, an
easy induction procedure allows us to finish the proof.
\end{proof}
\begin{corollary}\label{cor:interpolation-E(M)}  Let $E_j$ be a symmetric quasi Banach function space on $(0,\alpha)$ which is
$s_j$-convex for some $0 < s_j < \8,\; j=1,2$. If  $0<\theta<1$,
$0<p\le\8$ and $E=(E_1,E_2)_{\theta,p}$, then $E(\M)=(E_1(\M),E_2(\M))_{\theta,p}$ with equivalent norms.

\end{corollary}

Let $R$ be the Riesz projection (see \cite{B,MW1,R}). If  $E$ is a symmetric Banach function space on $(0,\alpha)$ with $1<p_{E} \le q_{E}<\8$ and $E$ has order continuous norm, then $R$ is a projection
from $E(\M)$ onto in $E(\A)$. Indeed, since $E$  has order continuous norm, $\M\cap L_1(\M)$ is dense in $E(\M)$. Let $1<p<p_{E} \le q_{E}<q <\8$. Then $R$ is bounded projection from $ L_{p}(\M)$
onto $ H_{p}(\M)$ and from $ L_{q}(\M)$
onto $ H_{q}(\A)$. By \cite[Theorem 4.8]{D}, we know that $R$ is a bounded map from $ E(\M)$
onto $E(\A)$.

Therefore,
 if  $E_j$ is a separable symmetric  Banach function space on $(0,\alpha)$ with  $1<p_{E_j} \le q_{E_j}<\8$ ($j=1,\;2$), then for any $x\in E_1(\A)+E_2(\A)$
\beq\label{k-funtion}
 K_t\big(x,E_1(\A),E_2(\A)\big)\le C K_t\big(x,E_1(\M),E_2(\M)\big),\quad \forall t>0,
\eeq
where $C>0$ is independent of $\alpha$.

\begin{lemma}\label{lem:sum-hardy-space} Let $E_j$ be a separable symmetric quasi Banach function space on $(0,\alpha)$ which is
$s_j$-convex for some $0 < s_j < \8\; (j=1,\;2)$. Suppose that $q_{E_j}<\8$, $s_j>r>0\; (j=1,\;2)$ and $\alpha<\8$.
\begin{enumerate}[\rm(i)]
 \item
If $x\in E_1(\M)+E_2(\M)$ is an invertible operator such that
$x^{-1}\in \M$, then there exist a unitary
$u\in\M$ and $a\in E_1(\A)+E_2(\A)$ such that $x=ua$
and $a^{-1}\in \A$.
\item
$$
E_1(\A)+E_2(\A)=\big(E_1(\M)+E_2(\M)\big)\cap H_r(\A).
$$
\end{enumerate}
\end{lemma}
\begin{proof} (i) First assume that $r>1$. It follows that  $E_j$ admits an equivalent
symmetric norm (see \cite[Corollary 3.5]{DDS}), endowed with this new norm,  $E_j$ is a separable symmetric Banach function space on $(0,\alpha)$ and $r\le p_{E_j}\le q_{E_j}<\8$. By \eqref{boydindex}, $1< p_{E_j^\times}\le q_{E_j^\times}<\8\; (j=1,2)$. Using Corollary 1 in \cite{BM}, we obtain that
$$
E_j(\A)^*=E_j^\times(\A),\quad j=1,\;2.
$$
Hence,
\beq\label{eq:dual-sum-hardy}
\big(E_1(\A)+E_2(\A)\big)^*=E_1^\times(\A)\cap E_2^\times(\A).
\eeq
Let $x\in E_1(\M)+E_2(\M)$
 with $x^{-1}\in \M$. Then  $x\in L_{r}(\M)$ with
$x^{-1}\in \mathcal{M}$. Using Theorem 3.1 of
\cite{BX}, we find a unitary $u\in\M$ and $a\in
H_{r}(\A)$ such that $x=ua$ and $a^{-1}\in
\mathcal{A}$. If  $a\notin E_1(\A)+E_2(\A)$, by \eqref{eq:dual-sum-hardy},  there is an
 $b\in  E_1^\times(\A)\cap E_2^\times(\A)$ such that $\tau(b^*a)\neq0$  and $\tau(b^*d)=0$ for all $d\in E_1(\A)+E_2(\A)$. Choose $p$ such that $1<p<\min\{p_{E_1^\times},p_{E_2^\times}\}$. It is clear  that $b^*\in  L_p(\M)$ and $\tau(b^*d)=0$ for all $d\in \A$. Applying Proposition 2 in \cite{Sai}, we get that  $y^*$ belongs to the norm closure of $\A_0$ in $L_p(\M)$, where $\mathcal{A}_0=\mathcal{A}\cap\ker\mathcal{E}$. Hence, by Corollary 2.2 in \cite{BX}, we
 obtain that $\tau(y^*a)=\tau(\E(y^*a))=\tau(\E(y^*)\E(a))=0$. This is a contradiction,  so
  $ a\in E_1(\A)+E_2(\A)$.

Now assume that $\frac{1}{2}<r\leq1$. If $x=v|x|$ is the
polar decomposition of $x$, then $v\in\M$ is a
unitary. Let $x=v|x|^\frac{1}{2}|x|^\frac{1}{2}=x_1x_2,$
where $x_1=v|x|^\frac{1}{2}$ and $x_2=|x|^\frac{1}{2}$. It is clear that $x_k^{-1}\in\M$. Since $E_1(\M)+E_2(\M)=(E_1+E_2)(\M)$ with equivalent norms (see Theorem \ref{K-functional for noncommutative}),
\beq\label{eq:sum(2)}
x_1,\;x_2\in(E_1+E_2)^{(2)}(\M).
\eeq

On the other hand,  from Theorem \ref{K-functional for noncommutative} and its proof follows that
$$
\begin{array}{rl}
K_1\big(x_k, E_1^{(2)}(\M),E_2^{(2)}(\M)\big)
&=K_1\big(|x|^\frac{1}{2}, E_1^{(2)}(\M),E_2^{(2)}(\M)\big)\\
&\le A^{-1}K_1\big(\mu(x)^\frac{1}{2},E_1^{(2)},E_2^{(2)}\big)\\
&\le A^{-1}2^\frac{1}{2}K_1\big(\mu(x),E_1,E_2\big)^\frac{1}{2}\\
&\le A^{-1}(2B)^\frac{1}{2}K_1\big(x, E_1(\M),E_2(\M)\big)^\frac{1}{2}<\8,
\end{array}
$$
and so $x_k\in E_1^{(2)}(\M)+E_2^{(2)}(\M)\; (k=1,\;2)$. By the first case, we get a
factorization
 $$x_2=u_1 a_2$$
with $u_1\in\M$ a unitary, $a_2\in
E_1^{(2)}(\A)+E_2^{(2)}(\A)$ such that $a_2^{-1}\in\A$. Repeating this argument, we again obtain
a same factorization for $x_{1}u_1$:
 $$
 x_{1}u_1=ua_1
 $$
with $u\in\M$ a unitary, $a_1\in
E_1^{(2)}(\A)+E_2^{(2)}(\A)$ such that $a_1^{-1}\in\A$. Hence, we get a
factorization:
 $$x=u a_1 a_2.$$
Set $a=a_1a_2$. By \eqref{eq:sum(2)} and \eqref{saito}, we get $a\in E_1(\A)+E_2(\A)$.

For the case  $\frac{1}{4}<r\le\frac{1}{2}$, we repeat the previous
argument to complete the proof. Using induction, we obtain  the desired result for the general case.

(ii) By Lemma 4.5 in \cite{D},
$$
 E_j(\M)\subset L_r(\M),\quad j=1,\;2.
 $$
It is clear that $E_1(\A)+E_2(\A)\subset\big(E_1(\M)+E_2(\M)\big)\cap H_r(\A)$.

Conversely, if  $x\in H_{r}(\A)\cap \big(E_1(\M)+E_2(\M)\big)$.
 Let $x=v|x|$ be the polar decomposition of $x$.  By Lemma 1.1 in \cite{Ju}, there exists a contraction $b\in\mathcal{M}$
such that $|x|=b(|x|+1)$. On the other hand,  $|x|+1\in E_1(\M)+E_2(\M)$ and
$(|x|+1)^{-1}\in \M$. Using (i), we obtain  a unitary $u\in\M$  and $h\in E_1(\A)+E_2(\A)$  such that
$|x|+1=uh$ and $h^{-1}\in\mathcal{A}$.
Hence,
$$
vbu=xh^{-1}\in H_r(\mathcal{A})\cap \mathcal{M}=\mathcal{A}.
$$
So, it follows that $x\in E_1(\A)+E_2(\A)$. This finishes the proof.
\end{proof}

\begin{lemma}\label{lem:connection-real-complex} Let $E_j$ be a separable symmetric quasi Banach  function space on $(0,\alpha)$ which is
$s_j$-convex for some $0 < s_j < \8\; (j=1,2)$. If $0<\theta<1$ and $\alpha<\8$, then
$$
(E_1(\M),E_2(\M))_\theta\subset(E_1(\M),E_2(\M))_{\theta,\8}
$$
and
$$
(E_1(\A),E_2(\A))_\theta\subset(E_1(\A),E_2(\A))_{\theta,\8}.
$$
\end{lemma}
\begin{proof} Let $(E,\;K_t(\cdot))=(E_1+E_2,\;K_t(\cdot, E_1,E_2))$. Since  $E_j$ is
separable $(j=1,\;2)$, it is clear that $(E,\;K_t(\cdot))$ is separable.  We claim that  there is an equivalent  quasi norm $\|\cdot\|_E$ on $E$ which is equivalent to $K_t(\cdot)$   so that $E$, endowed with the new quasi norm $\|\cdot\|_E$, is a symmetric quasi Banach function space.
Indeed, first assume that $E_1$ and $E_2$ are symmetric  Banach function spaces, then  $E_1,\;E_2,\; E\subset L_1(0,\alpha)+ L_\8(0,\alpha)$. Let $x,\;y\in E$ and $x\preccurlyeq y$. If $y=y_1+y_2,\; y_1\in E_1,\; y_2\in E_2$,   then by Proposition 3 in \cite{LS} (see also Proposition 4.10 in \cite{DDP3}), there there exist $x_1,x_2\in L_1(0,\alpha)+ L_\8(0,\alpha)$ such that $x=x_1+x_2$ and $x_j\preccurlyeq y_j$, $j=1,\;2$. On the other hand, $E_1$ and $E_2$ are fully symmetric  Banach function spaces, it follows that $x_j\in E_j$, $j=1,\;2$. Hence
$$
K_t(x)\le \|x_1\|_{E_1} +t\|x_2\|_{E_2}\le \|y_1\|_{E_1} +t\|y_2\|_{E_2}.
$$
so that $K_t(x)\le K_t(y)$, i.e., $(E,\;K_t(\cdot))$ is symmetric Banach function space.

Now we prove the claim for the general case. We chose $n\in\mathbb{N}$ such that $ns_j>1$, then  $E^{(n)}$ can be renormed as a symmetric Banach function  space, $(j=1,\;2)$. By the first case,
$K_{t}(\cdot, E_1^{(n)},E_2^{(n)})$ is a symmetric  norm on $E_1^{(n)}+E_2^{(n)}$, hence $K_{t^{\frac{1}{n}}}(|\cdot|^\frac{1}{n}, E_1^{(n)},E_2^{(n)})^n$ is a symmetric quasi norm on $E_1+E_2$. From the proof of Theorem \ref{K-functional for noncommutative}, we know that  $K_{t^{\frac{1}{n}}}(|\cdot|^\frac{1}{n}, E_1^{(n)},E_2^{(n)})^n$ and $K_t(\cdot, E_1,E_2)$ are equivalent. The claim is  proved.

 Since  $E_1$ and $E_2$ are separable and
$s$-convex ($s=\min\{s_1,s_2\}$),  $(E,\;K_t(\cdot))$ is separable and $s$-convex.  Therefore,  $(E,\|\cdot\|_E)$  is separable and $s$-convex. By Lemma \ref{lem:analytic convex},  $(E(\M),\|\cdot\|_E)$ is  analytically convex. Hence,  $(E_1(\M)+E_2(\M),\;K_t(\cdot, E_1(\M),E_2(\M))$ is  analytically convex. By \eqref{eq:kalton}, the couple $(E_1(\M), E_2(\M))$ satisfies the condition (h) in \cite{Pe} (see \cite[$\S$5]{Pe} or \cite[p. 21]{CPP}). From the result in $\S$5 in \cite{Pe}, it follows that
$(E_1(\M),E_2(\M))_\theta\subset(E_1(\M),E_2(\M))_{\theta,\8}$. Similarly,
$$
(E_1(\A),E_2(\A))_\theta\subset(E_1(\A),E_2(\A))_{\theta,\8}.
$$
\end{proof}

\begin{theorem}\label{pro:k-funtion1} Let $E_j$ be a separable symmetric quasi Banach function space on $(0,\alpha)$ which is
$s_j$-convex for some $0 < s_j < \8\; (j=1,\;2)$. If $q_{E_j}<\8\; (j=1,\;2)$, then
that for all $x\in E_1(\A)+E_2(\A)$ and  all $t>0$
 $$
 K_t(x, E(\A),E_2(\A))\le C K_t(x, E_1(\M),E_2(\M)),
$$
where  $C$ is independent of $\alpha$.
\end{theorem}
\begin{proof} First assume $\alpha<\8$. Let $r=\min\{s_1,s_2\}$. If $r>1$, then  $1<r\le p_{E_j} \le q_{E_j}<\8$ ($j=1,\;2$). Hence,  by \eqref{k-funtion}, the result holds.

Suppose that $\frac{1}{2}<r\leq1$.  Let $K_t(x, E_1(\M),E_2(\M))<1$. Put $w=|x|+\varepsilon$ (where $\varepsilon\|1\|_{E_1}<1$). By Lemma 1.1 in \cite{Ju}, there exists a contraction $v\in \M$ such that $x = vw$. It is clear that $K_t(w, E_1(\M),E_2(\M))<K$, where $K>0$ depends only on the quasi norm constant of $\|\cdot\|_{E_1}$. Since $w^\frac{1}{2}\in E_1^{(2)}(\M)+E_2^{(2)}(\M)$ and $w^{-\frac{1}{2}}\in\M$,
by (i) of Lemma \ref{lem:sum-hardy-space}, there exist  a unitary $u\in M $ and $a\in E_1^{(2)}(\A)+E_2^{(2)}(\A)$ such that $w^{\frac{1}{2}}=ua$ and $a^{-1}\in\A$. It follows that  $b=xa^{-1}\in H_r(\A)$. Noticing that $b=vwa^{-1}=vww^{-\frac{1}{2}}u=vw^{\frac{1}{2}}u\in E_1^{(2)}(\M)+E_2^{(2)}(\M)$, by (ii) of Lemma \ref{lem:sum-hardy-space}, we get $b=xa^{-1}\in E_1^{(2)}(\A)+E_2^{(2)}(\A)$. Applying Theorem \ref{K-functional for noncommutative} we get
\be
\begin{array}{rl}
K_{t^\frac{1}{2}}(a, E_1^{(2)}(\M),E_2^{(2)}(\M))&\le A^{-1}K_{t^\frac{1}{2}}(\mu(a), E_1^{(2)}+E_2^{(2)})\\
&=A^{-1}K_{t^\frac{1}{2}}(\mu(w)^\frac{1}{2}, E_1^{(2)}+E_2^{(2)})\\
&\le A^{-1}\Big(2K_t(\mu(w), E_1,E_2)\Big)^\frac{1}{2}\\
&\le A^{-1}(2B)^\frac{1}{2}\Big(K_t(w, E_1(\M),E_2(\M))\Big)^\frac{1}{2}<C'
\end{array}
\ee
and
$$
K_{t^\frac{1}{2}}(b, E_1^{(2)}(\M),E_2^{(2)}(\M))\le K_{t^\frac{1}{2}}(w^{\frac{1}{2}}, E_1^{(2)}(\M),E_2^{(2)}(\M))<C',
$$
where  $C'>0$ is independent of $\alpha$.
Hence by \eqref{k-funtion},
$$
K_{t^\frac{1}{2}}(a, E_1^{(2)}(\A),E_2^{(2)}(\A))<C'C\qquad\mbox{and}\qquad K_{t^\frac{1}{2}}(b, E_1^{(2)}(\A),E_2^{(2)}(\A))<C'C.
$$
Therefore, there are $ y_1,y_2\in E_1^{(2)}(\A)$ and $z_1,z_2\in E_2^{(2)}(\A)$ such that
\beq\label{k-fun1}
b=y_1+z_1,\quad a=y_2+z_2,\quad\|y_i\|_{E_1^{(2)}}+t^{\frac{1}{2}}\|z_i\|_{E_2^{(2)}}<C'C,\quad i=1,\;2.
\eeq
On the other hand, we have $x=ba=(y_1+z_1)(y_2+z_2)=y_1y_2+z_1z_2+y_1z_2+z_1y_2$. Hence
\beq\label{k-fun2}\begin{array}{rl}
K_{t}(x, E_1(\A)+E_2(\A))&\le K' K_{t}(y_1y_2+z_1z_2, E_1(\A)+E_2(\A))\\
&\qquad+K'K_{t}(y_1z_2+z_1y_2, E_1(\A)+E_2(\A)),
\end{array}
\eeq
where $K'>0$ depends only on the quasi norm constants of $\|\cdot\|_{E_1}$ and $\|\cdot\|_{E_2}$.
By \eqref{k-fun1} and Proposition \ref{pro:noncommutative-ponitwise product}, we get
\beq\label{k-fun3}
 K_{t}(y_1y_2+z_1z_2, E_1(\A)+E_2(\A))\le \|y_1y_2\|_{E_1}+t\|z_1z_2\|_{E_2}
 <C''(C'C)^2,
\eeq
where  $C''>0$ is independent of $\alpha$.
Using Corollary \ref{cor:product-hardy} and  \eqref{k-fun1}, we obtain that
 \beq\label{k-fun4}
 \begin{array}{rl}
 \|y_1z_2+z_1y_2\|_{(E_1(\A),E_2(\A))_\frac{1}{2}}&\le B'(\|y_1\|_{E_1^{(2)}}\|z_2\|_{E_2^{(2)}}+\|z_1\|_{E_1^{(2)}}\|y_2\|_{E_2^{(2)}})\\
 &<2B'(C'C)^2t^{-\frac{1}{2}}.
\end{array}
\eeq
where  $B'>0$ is independent of $\alpha$. Applying
 Lemma \ref{lem:connection-real-complex} and \eqref{k-fun4}, we obtain that
$$
\|y_1z_2+z_1y_2\|_{(E_1(\A),E_2(\A))_{\frac{1}{2},\8}}\le 2B'A'(C'C)^2t^{-\frac{1}{2}},
$$
where  $A'>0$ depends on the constant of $(E_1(\A),E_2(\A))_\theta\subset(E_1(\A),E_2(\A))_{\theta,\8}$.
This implies
\beq\label{k-fun5}
 K_{t}(y_1z_2+z_1y_2, E_1(\A),E_2(\A))\le 2B'A'(C'C)^2.
\eeq
Combining \eqref{k-fun2}, \eqref{k-fun3} and \eqref{k-fun5}, we find that
$$
K_{t}(x, E_1(\A),E_2(\A))\le L,
$$
where $L>0$ is independent of $\alpha$. The remainder of the proof can be done the same way as in the proof of Theorem \ref{K-functional for noncommutative}.

Finally, suppose $\alpha=\8$. By the finite case,  for all $i\in I$ and $t>0$,
$$
 K_t(x, E_1(\A_{e_i}),E_2(\A_{e_i}))\le L K_t(x, E_1(\M_{e_i}),E_2(\M_{e_i})),\qquad x\in E_1(\A_{e_i})+E_2(\A_{e_i}).
 $$
Let $x\in  E_1(\A)+E_2(\A)$. Then there exist  $x_1\in E_1(\A)$ and $x_2\in E_2(\A)$ such that
$x=x_1+x_2$. Hence,
$$
\begin{array}{rl}
 &K_t(x,  E_1(\A),E_2(\A)) \\
 &\qquad \le C[ K_t( e_ix e_i,  E_1(\A),E_2(\A)))+ K_t(x- e_ixe_i,  E_1(\A),E_2(\A))]\\
 &\qquad\le C[K_t( e_ix e_i,  E_1(\A_{e_i}),E_2(\A_{e_i}))+ \|x_1-e_ix_1e_i\|_{E_1}+ t\|x_2-e_ix_2e_i\|_{E_2}]\\
& \qquad\le CLK_t( e_ix e_i, E_1(\M_{e_i}),E_2(\M_{e_i}))+ C[\|x_1-e_ix_1e_i\|_{E_1}+ t\|x_2-e_ix_2e_i\|_{E_2}]\\
  &\qquad\le CLK_t( x, E_1(\M),E_2(\M))+ C[\|x_1-e_ix_1e_i\|_{E_1}+ t\|x_2-e_ix_2e_i\|_{E_2}].\\
\end{array}
$$
On the other hand, by Lemma \ref{semi-closure}, $\|x_1-e_ix_1e_i\|_{E_1}\rightarrow0$ and $\|x_2-e_ix_2e_i\|_{E_2}\rightarrow0$. Therefore,
 $$
 K_t(x, E(\A),E_2(\A))\le CL K_t(x, E_1(\M),E_2(\M)).
$$

\end{proof}

\begin{corollary}\label{cor:interpolation-E(A)} Let $E_j$ be a separable symmetric quasi Banach function space on $(0,\alpha)$ which is
$s_j$-convex for some $0 < s_j < \8$ and $q_{E_j}<\8\;(j=1,\;2)$. If  $0<\theta<1$,
$0<p\le\8$ and $E=(E_1,E_2)_{\theta,p}$, then $E(\A)=(E_1(\A),E_2(\A))_{\theta,p}$.

\end{corollary}

\subsection*{Acknowledgement} We thank the  referees for very useful comments, which improved the paper. We also thank the  referee for  suggesting (pointing out) the proofs of Proposition \ref{pro:calderon} and \ref{pro:product}.  \\
 T.N. Bekjan is partially supported by NSFC grant  No.11771372, M.N. Ospanov is partially supported by  project AP05131557 of the Science Committee of Ministry of Education and Science of the Republic of Kazakhstan.

\end{document}